\documentclass{amsart}

\usepackage{amsthm}
\usepackage{graphicx}
\usepackage{stmaryrd} 
\usepackage{tikz}
\usepackage{tikz-cd}
\usetikzlibrary{arrows}
\usepackage{verbatim}
\usepackage{amssymb,amsfonts}
\usepackage[all,arc]{xy}
\usepackage{enumerate}
\usepackage{mathrsfs, mathabx}
\usepackage{xcolor}[2007/01/21]
\usepackage{hyperref}
\usepackage{soul}
\usepackage{esint} 
\usepackage{multirow} 
\usepackage[all]{xy}
\usepackage[margin=1.0in]{geometry}
\usepackage[
   open,
   openlevel=2,
   atend
 ]{bookmark}[2011/12/02]
\usepackage{graphics}

\bookmarksetup{color=blue}

\theoremstyle{definition} 
\newtheorem{thm}{Theorem}[section]
\newtheorem{cor}[thm]{Corollary}
\newtheorem{prop}[thm]{Proposition}
\newtheorem{lem}[thm]{Lemma}

\theoremstyle{definition}

\newtheorem{exmp}[thm]{Example}

\newtheorem{rmk}[thm]{Remark}

\theoremstyle{definition}

\theoremstyle{remark}

\newcommand{\Conf}{\mathrm{Conf}}

\newcommand{\bZ}{\mathbb{Z}}

\newcommand{\bR}{\mathbb{R}}
\newcommand{\bP}{\mathbb{P}}
\newcommand{\bA}{\mathbb{A}}
\newcommand{\bC}{\mathbb{C}}
\newcommand{\bF}{\mathbb{F}}

\newcommand{\mb}{\mathbf}
\newcommand{\mr}{\mathrm}

\newcommand{\sm}{\smallsetminus}
\newcommand{\sub}{\subset}

\newcommand{\Spec}{\mathrm{Spec}}

\newcommand{\Sym}{\mathrm{Sym}}

\newcommand{\Prob}{\mathrm{Prob}}

\newcommand{\om}{\omega}

\newcommand{\Hom}{\mathrm{Hom}}

\newcommand{\id}{\mathrm{id}}
\newcommand{\bs}{\boldsymbol}

\newcommand{\es}{\emptyset}
\newcommand{\bl}{\bullet}

\newcommand{\ra}{\rightarrow}

\newcommand{\hra}{\hookrightarrow}

\newcommand{\ot}{\otimes}

\newcommand{\bop}{\bigoplus}

\newcommand{\be}{\begin{enumerate}}
\newcommand{\ee}{\end{enumerate}}

\newcommand{\Ep}{E^{\times}}
\newcommand{\Gr}{\mathrm{Gr}}
\newcommand{\Dl}{\Delta}
\newcommand{\wg}{\wedge}
\newcommand{\sg}{\sigma}

\numberwithin{equation}{section}

\begin{document}

\title[]{Betti and Hodge numbers of configuration spaces of a punctured elliptic curve from its zeta functions}
\author{Gilyoung Cheong and Yifeng Huang}
\address{
Cheong: Department of Mathematics, University of California--Irvine, 340 Rowland Hall, Irvine, CA 92697-3875 \\
Huang: Department of Mathematics, University of Michigan, 530 Church Street, Ann Arbor, MI 48109-1043}
\email{gilyounc@uci.edu, huangyf@umich.edu}

\begin{abstract}
Given an elliptic curve $E$ defined over $\bC$, let $\Ep$ be an open subset of $E$ obtained by removing a point. In this paper, we show that the $i$-th Betti number of the unordered configuration space $\Conf^{n}(\Ep)$ of $n$ points on $\Ep$ appears as a coefficient of an explicit rational function in two variables. We also compute its Hodge numbers as coefficients of another explicit rational function in four variables. Our result is interesting because these rational functions resemble the generating function of the $\bF_{q}$-point counts of $\Conf^{n}(\Ep)$, which can be obtained from the zeta function of $E$ over a finite field $\bF_{q}$. We show that the mixed Hodge structure of the $i$-th singular cohomology group $H^{i}(\Conf^{n}(\Ep))$ with complex coefficients is pure of weight $w(i)$, an explicit integer we provide in this paper. This purity statement implies our main result about the Betti numbers and the Hodge numbers. Our proof uses Totaro's spectral sequence computation that describes the weight filtration of the mixed Hodge structure on $H^{i}(\Conf^{n}(\Ep))$.
\end{abstract}

\maketitle

\section{Introduction}

\subsection{Motivation and main results} In number theory, it is classically known that the probability that a random positive integer is square-free is equal to $1/\zeta_{\Spec(\bZ)}(2)$, where $\zeta_{\Spec(\bZ)}(s)$ is the Riemann zeta function. More specifically, we have

$$\lim_{n \ra \infty}\underset{\substack{m \in \bZ: \\ 1 \leq m \leq n}}{\Prob}(m \text{ is square-free}) = \frac{1}{\zeta_{\Spec(\bZ)}(2)}.$$

\

The analogue for $\bF_{q}[x]$ in place of $\bZ$, where $\bF_{q}$ is a finite field, is even better: we have

$$\underset{\substack{f \in \bF_{q}[x]: \\\deg(f) = n}}{\Prob}(f \text{ is square-free}) = 1 - q^{-1} = \frac{1}{\zeta_{\bA^{1}_{\bF_{q}}}(2)},$$

\

for any fixed degree $n \geq 2$, where $\zeta_{\bA^{1}_{\bF_{q}}}(s)$ is the zeta function of the affine line $\bA^{1}_{\bF_{q}} = \Spec(\bF_{q}[x])$ over $\bF_{q}$.  In \cite{CEF}, Church, Ellenberg, and Farb explains how counting such polynomials is related to the topology of the unordered configuration space $\Conf^{n}(\bC)$ of the complex plane $\bC = \bR^{2}$, or equivalently, the set of $\bC$-points on the affine line over $\bC$. After realizing the set of monic square-free polynomials in $\bF_{q}[x]$ as the set $\Conf^{n}(\bA^{1})(\bF_{q})$ of the $\bF_{q}$-points on the unordered configuration space of the affine line, they show that

$$|\Conf^{n}(\bA^{1})(\bF_{q})| = \sum_{i=0}^{\infty}(-1)^{i}q^{n-i}h^{i}(\Conf^{n}(\bC)),$$

\

where $h^{i}$ means the $i$-th Betti number, using the $l$-adic cohomology theory along with the result about how the geometric Frobenius acts on the $i$-th $l$-adic cohomology group of an arbitrary affine hyperplane complement over $\bF_{q}$, independently known by Lehrer \cite{Leh} and Kim \cite{Kim}. Note that

\begin{align*}
\sum_{n=0}^{\infty}|\Conf^{n}(\bA^{1})(\bF_{q})|t^{n} &= \sum_{\substack{f \in \bF_{q}[x]: \\f \text{square-free}}} t^{\deg(f)} = \prod_{P \in |\bA^{1}_{\bF_{q}}|}(1 + t^{\deg(P)}) = \prod_{P \in |\bA^{1}_{\bF_{q}}|}\frac{1 -  t^{2\deg(P)}}{1 - t^{\deg(P)}} = \frac{Z_{\bA^{1}_{\bF_{q}}}(t)}{Z_{\bA^{1}_{\bF_{q}}}(t^{2})} = \frac{1 - qt^{2}}{1 - qt}
\end{align*}

\

where $|\bA^{1}_{\bF_{q}}|$ is the set of monic irreducible polynomials in $\bF_{q}[x]$ and $Z_{\bA^{1}_{\bF_{q}}}(t)$ is the zeta series of $\bA^{1}_{\bF_{q}}$, meaning $Z_{\bA^{1}_{\bF_{q}}}(q^{-s}) = \zeta_{\bA^{1}_{\bF_{q}}}(s)$. Hence, we have

$$\sum_{n=0}^{\infty}\sum_{i=0}^{\infty}(-1)^{i}q^{n-i}h^{i}(\Conf^{n}(\bC))t^{n} = \frac{1 - qt^{2}}{1 - qt} = 1 + qt + (q^{2} - q)t^{2} + (q^{3} - q^{2})t^{3} + \cdots.$$

\

Since the above identity holds for all prime powers $q$, we have

$$h^{i}(\Conf^{n}(\bC)) = \left\{
	\begin{array}{ll}
	1 & \mbox{if } n = 0, 1 \text{ and } i = 0,  \\
	1 & \mbox{if } n \geq 2 \text { and } i = 0, 1, \\
	0 & \mbox{otherwise}, \\
	\end{array}\right.$$

\

which recovers a result of Arnol'd in \cite{Arn}.

\

\hspace{3mm} In this paper, we seek a genus $1$ analogue of this story by replacing $\bA^{1}$, or $\bP^{1}$ minus a (degree 1) point, by an elliptic curve $E$ minus a point, which can be defined by any equation of the form

$$y^{2} = f(x),$$

\

where $f(x)$ is a square-free polynomial of degree $3$ over the ambient field. We denote this punctured elliptic curve as $E^{\times}$. Over $\bF_{q}$, a theorem of Weil gives an explicit form of the zeta series

$$Z_{E^{\times}}(t) = \frac{(1 - \alpha t)(1 - \bar{\alpha} t)}{1 - qt},$$

\

where $\alpha$ is an algebraic integer with the complex norm $q^{1/2}$ and $\bar{\alpha}$ its complex conjugation. This implies that

\begin{align*}
\sum_{n=0}^{\infty}|\Conf^{n}(E^{\times})(\bF_{q})|t^{n} = \frac{Z_{\Ep}(t)}{Z_{\Ep}(t^{2})} = \frac{(1 - \alpha t)(1 - \bar{\alpha} t)(1 - qt^{2})}{(1 - \alpha t^{2})(1 - \bar{\alpha} t^{2})(1 - qt)},
\end{align*}

\

using a similar computation for the $\bA^{1}$ case. Unfortunately, the proof presented in \cite{CEF} does not generalize to connect this arithmetic result to topology. As our main result, we obtain a topological analogue of the above arithmetic computation with a different proof:

\begin{thm}\label{main1} Let $\Ep$ be an open subset of an elliptic curve over $\bC$ obtained by removing a point, with respect to the analytic topology. We have

$$\sum_{n=0}^{\infty}\sum_{i=0}^{\infty}(-1)^{i}h^{i}(\Conf^{n}(E^{\times})) u^{2n - w(i)}t^{n} = \frac{(1 - ut)^{2}(1 - u^{2}t^{2})}{(1 - ut^{2})^{2}(1 - u^{2}t)}$$

\

and 

\

$$\sum_{n=0}^{\infty} \sum_{i=0}^{\infty} (-1)^i \sum_{p,q\geq 0} h^{n-p,n-q}(H^{i}(\Conf^{n}(\Ep))) x^p y^q u^{2n - w(i)} t^n = \frac{(1-xut)(1-yut)(1-xyu^{2}t^2)}{(1-xut^2)(1-yut^2)(1-xyu^{2}t)},$$

\

where 

$$w(i) := \left\{
	\begin{array}{ll}
	3i/2 & \mbox{if } i \mbox{ is even and}  \\
	(3i-1)/2 & \mbox{if } i \mbox{ is odd.}
	\end{array}\right.$$
	
\

and $h^{p,q}(H^{i}(\Conf^{n}(\Ep)))$ denote the Hodge numbers of the $i$-th cohomology group\footnote{From now on, every cohomology of a complex variety we deal with will be assumed to be its singular cohomology with complex coefficients with respect to the analytic topology. Furthermore, every variety will be assumed to be defined over complex numbers, unless mentioned otherwise, and we will identify it with its set of complex points.} of $\Conf^{n}(\Ep)$. In particular, since $w : \bZ_{\geq 0} \ra \bZ_{\geq 0}$ is injective, the rational function computes all the Hodge numbers.
\end{thm}

\

\begin{rmk} We note that $\Conf^{n}(E^{\times})$ is smooth but not projective (nor proper), so to discuss its Hodge theory, we need Deligne's mixed Hodge structure, introduced in \cite{Del}. Even the first statement of Theorem \ref{main1} about the Betti numbers of $\Conf^{n}(E^{\times})$ will be deduced from the following, a vanishing statement about the Hodge numbers, which is our major contribution. Thus, the computation of the Betti numbers can be regarded as a concrete application of Deligne's theory.
\end{rmk}

\

\begin{thm}\label{main2} For any $i \in \bZ_{\geq 0}$, the mixed Hodge structure of the $i$-th cohomology group $H^{i}(\Conf^{n}(\Ep))$ of $\Conf^{n}(\Ep)$ is pure of weight $w(i)$, the number mentioned in Theorem \ref{main1}. In other words, we have $h^{p,q}(H^{i}(\Conf^{n}(\Ep))) = 0$ unless $p + q = w(i)$.
\end{thm}

\

\begin{rmk} There are direct genus $0$ analogues of Theorem \ref{main1} and Theorem \ref{main2} with $\bP^{1}(\bC)$ replacing $E$ and thus with $\bA^{1}(\bC) = \bC$ replacing $\Ep$. It follows from \cite[Theorem 1]{Kim} that the $i$-th cohomology group of $\Conf^{n}(\bC)$ has a pure Hodge structure of weight $2i$, meaning that $h^{p,q}(H^{i}(\Conf^{n}(\bC))) = 0$ unless $p + q = 2i$. This is the genus $0$ analogue of Theorem \ref{main2}, which is the genus $1$ case. Just as our paper will deduce Theorem \ref{main1} from Theorem \ref{main2}, this genus $0$ purity result implies that 

\begin{align*}
\sum_{n=0}^{\infty}\sum_{i=0}^{\infty}(-1)^{i}h^{i}(\Conf^{n}(\bC)) u^{2n - 2i}t^{n} &= \frac{1 - u^{2}t^{2}}{1 - u^{2}t} \\
&= 1 + u^{2}t + (u^{4} - u^{2})t^{2} + (u^{6} - u^{4})t^{3} + \cdots,
\end{align*}

\

which is analogous to the first part\footnote{One may also deduce the analogue of the second part of Theorem \ref{main2}, but we omit this discussion.} of Theorem \ref{main1}. One may note that replacing $u^{2}$ with $q$ recovers the generating function for $|\Conf^{n}(\bA^{1})(\bF_{q})|$, introduced in the beginning.

\

\hspace{3mm} For the case of $\Conf^{n}(\bC)$, in contrast to our case of $\Conf^{n}(\Ep)$, as seen in \cite{Kim}, the $i$-th cohomology group $H^{i}(F(\bC, n))$ of the ordered configuration space $F(\bC, n)$ of $n$ points on $\bC$ has a pure Hodge structure of weight $2i$, which implies that $H^{i}(\Conf^{n}(\bC))$ also has a pure Hodge structure of weight $2i$. However, it turns out that $H^{i}(F(\Ep, n))$, replacing $\bC$ with $\Ep$, does not have a pure Hodge structure in general (e.g., Example 4.2 of \cite{Bib}), so our approach to Theorem \ref{main2} has to be significantly different. We use the fact that the $S_{n}$-invariance of the Leray spectral sequence $\{E_{p,q}^{r}\}$ of the inclusion

$$F(\Ep, n) \hra (\Ep)^{n}$$

\

converges to the cohomology group we want to compute and degenerates at the $E_{3}$ page, which turns out to remember the weight filtration of the cohomology. (Recall that for showing the purity, understanding the weight filtration is enough.) Moreover, the $E_{2}$ page of this spectral sequence can be explicitly described using a work of Totaro in \cite{Tot}, so we use this description, which we review in Section \ref{totaro}, to approach our problem in an explicit manner.
\end{rmk}

\

\subsection{Related works and future directions} The literature of cohomology of the (either ordered or unordered) configuration spaces of a manifold is extensive, so we refer to only a few of them directly related to our work. The Betti numbers $h^{i}(\Conf^{n}(E^{\times}))$ were first computed by B\"odigheimer and Cohen \cite{BC}. (See \cite[Proposition 3.5]{DK} for more explicit computations), and it was also studied by by Napolitano \cite[p.489, Table 3]{Nap}. The new features of Theorem \ref{main1} are the computation of Hodge numbers $h^{p,q}(H^{i}(\Conf^{n}(E^{\times})))$ and the explicit rational generating functions for either Betti or Hodge numbers. Jesse Wolfson pointed out that the rationality we exhibit comes from the fact that we take the $S_{n}$-invariants (i.e., the main result in \cite{Mac} works for any graded symmetric algebra), although explicitly computing such rational functions may require more, just as we used purity (Theorem \ref{main2}) in our paper. A natural question is to ask if we can compute Hodge numbers in other cases, which seems unclear as our proof seems quite particular to the case we deal with. The problem also does not restrict to configuration spaces; for example, see Farb, Wolfson, and Wood \cite[Theorem 3.1, Statement 3]{FWW}, which is a generalization of Totaro's work in \cite{Tot}.

\

\subsection{Organization of the paper} In Section \ref{implication}, we explain why the purity statement given by Theorem \ref{main2} implies Theorem \ref{main1}. We review Totaro's work in Section \ref{totaro}, which is a key tool for our proof of Theorem \ref{main2}. In Section \ref{proof}, we prove Theorem \ref{main2}.

\

\subsection{Acknowledgments} We thank Christin Bibby, Haoyang Guo, James Hotchkiss, Mircea Musta\c{t}\u{a}, Will Sawin, John Stembridge, Burt Totaro, and Jesse Wolfson for helpful discussions. Huang was supported by Research Training Grant (RTG): Number Theory and Representation Theory at the University of Michigan while completing this work.

\

\section{Theorem \ref{main2} implies Theorem \ref{main1}} \label{implication}

\hspace{3mm} In this section, we explain how Theorem \ref{main2} implies Theorem \ref{main1}. Our work depends on the mixed Hodge structure on the cohomology of algebraic varieties, and we use some of their known properties.\footnote{These properties are reviewed, for example, in a paper by Danilov and Khovanski \cite[Section 1]{DK86}.} Consider the polynomial

$$\chi(X, u) := \sum_{i, j = 0}^{\infty} (-1)^{i} \dim_{\bC}(\Gr_{W}^{j}(H^{i}_{c}(X)))u^{j},$$

\

where $H^{i}_{c}(X)$ means the $i$-th compactly supported (singular) cohomology group of a variety $X$ with complex coefficients and $\Gr_{W}^{j}(H^{i}_{c}(X))$ means its $j$-th graded quotient with respect to the weight filtration on $H^{i}_{c}(X)$.

\

\begin{exmp} If $X$ is smooth and projective, then the Hodge structure on $H^{i}_{c}(X) = H^{i}(X)$ is pure of weight $i$.\footnote{In this case, we simply say that ``$H^{i}_{c}(X)$ is pure of weight $i$,'' and similarly for other weights for any pure Hodge structure.} In other words, we have

$$\Gr^{j}_{W}(H^{i}_{c}(X)) = \left\{
	\begin{array}{ll}
	H^{i}_{c}(X) & \mbox{if } j = i \mbox{ and}  \\
	0 & \mbox{if } j \neq i.
	\end{array}\right.$$

\

Thus, in this case, we have

$$\chi(X, u) = \sum_{i = 0}^{\infty}(-1)^{i} h^{i}_{c}(X)u^{i} = \sum_{i = 0}^{\infty}(-1)^{i} h^{i}(X)u^{i},$$

\

where $h^{i}_{c}(X) := \dim_{\bC}(H^{i}_{c}(X))$. This implies that $\chi(X, 1) = \chi(X)$, the Euler characteristic of $X$. For our purpose, we will only deal with varieties that are smooth, but not necessarily projective nor proper.
\end{exmp}

\

\hspace{3mm} For any closed subvariety $Z$ of any variety $X$, it is known (e.g., from \cite[p.198]{FM}) that

$$\chi(X, u) = \chi(Z, u) + \chi(X \sm Z, u).$$

\

Since the weight filtration on $H^{i}_{c}(X)$ is compatible with the K\"unneth formula, we also get

$$\chi(X \times Y, u) = \chi(X, u)\chi(Y, u),$$

\

for any two varieties $X$ and $Y$. We will call these two properties the \textbf{motivic properties} of $\chi(-, u)$. Given a quasi-projective variety $X$, the $n$-th symmetric power $\Sym^{n}(X) = X^{n}/S_{n}$ exists as a variety for every $n \in \bZ_{\geq 0}$, whose set of $\bC$-points corresponds to the quotient topological space. We define two power series

$$Z(X, u, t) := \sum_{n=0}^{\infty}\chi(\Sym^{n}(X), u)t^{n}$$

\

and

$$K(X, u, t) := \sum_{n=0}\chi(\Conf^{n}(X), u)t^{n}.$$

\

Due to the motivic properties of $\chi(-, u)$, it follows from Proposition 5.9 of \cite{VW} that

$$K(X, u, t) = \frac{Z(X, u, t)}{Z(X, u, t^{2})}.$$

\

\

\subsection{Theorem \ref{main2} implies the first part of Theorem \ref{main1}} Consider a smooth quasi-projective variety $X$. If $H^{i}(X)$ is pure of weight $w(i)$, possibly different from $i$, then $H^{2n-i}_{c}(X)$ is pure of weight $2n - w(i)$, so

\begin{align*}
\chi(X, u) &= \sum_{i = 0}^{\infty}(-1)^{i} h^{2n-i}_{c}(X)u^{2n-w(i)} \\
&= \sum_{i = 0}^{\infty}(-1)^{i} h^{i}(X)u^{2n-w(i)},
\end{align*}

\

where $n$ is the (complex) dimension of $X$. Hence, if we assume Theorem \ref{main2}, the first part of Theorem \ref{main1} merely says

$$K(\Ep, u, t) = \frac{(1 - ut)^{2}(1 - u^{2}t^{2})}{(1 - ut^{2})^{2}(1 - u^{2}t)}.$$

\

Using a long exact sequence and Poincar\'e duality, we observe that $H^{i}(\Ep) \simeq H^{i}(E)$ for $i = 0, 1$ and that these isormophisms preserve the mixed Hodge structures.\footnote{See \cite[p.282]{DK86}.} Since $H^{i}(\Ep) = 0$ for $i \geq 2$, we see that each $H^{i}(\Ep)$ is pure of weight $i$ so that $H^{i}((\Ep)^{n})$ is pure of weight $i$, using the fact that the K\"uneeth isomorphism preserves the mixed Hodge structures. Thus, it follows that $H^{2n-i}_{c}((\Ep)^{n})$ is pure of weight $2n - i$ so that $H^{i}_{c}((\Ep)^{n})$ is pure of weight $i$. This implies that $H^{i}_{c}(\Sym^{n}(\Ep))$ is pure of weight $i$ because

$$H^{i}_{c}(\Sym^{n}(\Ep)) \simeq H^{i}_{c}((\Ep)^{n})^{S_{n}} \hra H^{i}_{c}((\Ep)^{n})$$

\

induced by the quotient map $(\Ep)^{n} \ra (\Ep)^{n}/S_{n} = \Sym^{n}(\Ep)$, which is a finite map, and all the maps above are strictly compatible with the mixed Hodge structures. Thus, we have

$$\chi(\Sym^{n}(\Ep), u) = \sum_{i=0}^{\infty}(-1)^{i} h^{i}_{c}(\Sym^{n}(\Ep))u^{i},$$

\

and by a formula due to Macdonald (originally from \cite{Mac} but we use the version given in Cheah's thesis \cite[p.116]{Che} for the compactly supported cohomology), we have

\begin{align*}
Z(\Ep, u, t) &= \sum_{n=0}^{\infty}\chi(\Sym^{n}(\Ep), u) t^{n} \\
&= \frac{(1 - ut)^{h^{1}_{c}(\Ep)}}{(1 - t)^{h^{0}_{c}(\Ep)}(1 - u^{2}t)^{h^{2}_{c}(\Ep)}} \\
&= \frac{(1 - ut)^{h^{1}(\Ep)}}{(1 - t)^{h^{2}(\Ep)}(1 - u^{2}t)^{h^{0}(\Ep)}} \\
&= \frac{(1 - ut)^{2}}{1 - u^{2}t}.
\end{align*}

\

Therefore, we have

$$K(\Ep, u, t) = \frac{Z(\Ep, u, t)}{Z(\Ep, u, t^{2})} = \frac{(1 - ut)^{2}(1 - u^{2}t^{2})}{(1 - u^{2}t)(1 - ut^{2})^{2}},$$

\

so this proves that Theorem \ref{main2} implies the first part of Theorem \ref{main1}.

\

\subsection{Theorem \ref{main2} implies the second part of Theorem \ref{main1}} For any variety $X$, consider the mixed Hodge polynomial 

$$\mb{H}_{c}(X,x,y,u):=\sum_{p,q,i \geq 0} h^{p,q}(H^{i}_{c}(X))) x^{p} y^{q} (-u)^{i}.$$

\

This is a generating function for the Hodge numbers

$$h^{p,q}(H^{i}_{c}(X))) := \dim_{\bC}(\Gr^{F}_{p}\Gr_{W}^{p+q}(H^{i}_{c}(X)))$$

\

where ${\Gr}^{F}_{p}$ denotes the $p$-th graded piece of the Hodge filtration. Consider the following generating function:

$$\mb{Z}(X,x,y,u,t) := \sum_{n=0}^{\infty} \mb{H}_{c}(\Sym^n(X),x,y,u) t^{n},$$

\

defined for any quasi-projective variety $X$. Using the fact that the Hodge numbers of $E$ are symmetric (i.e., $h^{p,q}(E) = h^{q,p}(E)$), we can compute each Hodge number $h^{p,q}(E)$. From this and a long exact sequence, we can compute the Hodge numbers of $\Ep$ so that
 
$$\mb{H}_{c}(\Ep,x,y,u)=-(x+y)u+xyu^2.$$

\

Hence, using a formula due to Cheah \cite[p.116]{Che}, we obtain that the generating function for $X = \Ep$ is given by

$$\mb{Z}(\Ep,x,y,u,t)= \frac{(1-xut)(1-yut)}{(1-xyu^2t)}.$$

\

Recall that $\mb{H}_{c}(-,x,y,1)$ satisfies the motivic properties (e.g., from \cite[Propositions 1.6 and 1.8]{DK86}). By \cite[Proposition 5.9]{VW}, we have

\begin{align*}
\sum_{n=0}^\infty \chi(\Conf^{n}(\Ep),x,y,1) t^{n} &= \frac{\mathbf{Z}(\Ep,x,y,1,t)}{\mathbf{Z}(\Ep,x,y,1,t^{2})} \\
&= \frac{(1-xt)(1-yt)(1-xyt^2)}{(1-xyt)(1-xt^2)(1-yt^2)}.
\end{align*}

\

This implies that

\begin{align*}
\sum_{n=0}^{\infty} \sum_{p,q\geq 0}\sum_{i = 0}^{\infty}  (-1)^{i} h^{p,q}(H^{2n-i}_{c}(\Conf^{n}(\Ep))) x^p y^q t^n &= \sum_{n=0}^{\infty} \sum_{p,q\geq 0}\sum_{i = 0}^{\infty}  (-1)^{i} h^{p,q}(H^{i}_{c}(\Conf^{n}(\Ep))) x^p y^q t^n \\
&= \frac{(1-xt)(1-yt)(1-xyt^2)}{(1-xyt)(1-xt^2)(1-yt^2)}.
\end{align*}

\

We now use Theorem \ref{main2}, which tells us that $h^{p,q}(H^{2n-i}_{c}(\Conf^{n}(\Ep))) = h^{n-p,n-q}(H^{i}(\Conf^{n}(\Ep))) = 0$ unless $p + q = 2n - w(i)$, so

$$\sum_{n=0}^{\infty} \sum_{\substack{p,q\geq 0, \\ p+q = 2n-w(i)}}\sum_{i = 0}^{\infty}  (-1)^{i} h^{n-p,n-q}(H^{i}(\Conf^{n}(\Ep))) x^p y^q t^n = \frac{(1-xt)(1-yt)(1-xyt^2)}{(1-xyt)(1-xt^2)(1-yt^2)}.$$

\

Thus, we may replace $x$ and $y$ by $xu$ and $yu$ respectively to get

$$\sum_{n=0}^{\infty} \sum_{\substack{p,q\geq 0, \\ p+q = 2n-w(i)}}\sum_{i = 0}^{\infty}  (-1)^{i} h^{n-p,n-q}(H^{i}(\Conf^{n}(\Ep))) x^{p} y^{q}u^{2n-w(i)} t^{n} = \frac{(1-xut)(1-yut)(1-xyu^{2}t^2)}{(1-xyut)(1-xut^{2})(1-yut^{2})},$$

\

and the above identity holds without specifying the conditions $p+q = 2n-w(i)$, because all the coefficients on the left-hand side violating such conditions are equal to $0$ by Theorem \ref{main2}. Thus, we obtain the second part of Theorem \ref{main1} from Theorem \ref{main2}.

\

\section{Totaro's description of a Leray spectral sequence} \label{totaro}

\subsection{Recalling Poincar\'e duality} We recall the following explicit form of the Poincar\'e duality (e.g., from \cite[Theorem 24.18]{Ful}). That is, given any oriented real manifold $M$ of dimension $m$, the Poincar\'e duality is given by the $\bR$-linear isomorphism

$$H^{i}(M, \bR) \simeq H^{m - i}_{c}(M, \bR)^{\vee}$$

\

given by

$$[\omega] \mapsto \left([\mu] \mapsto \int_{M} \omega \wedge \mu \right),$$

\

where $H^{i}(M, \bR)$ denotes the $i$-th de Rham cohomology of $M$ with real coefficients and similarly for the compactly supported de Rham cohomology. We wrote $V^{\vee} := \Hom_{k}(V, k)$ for the dual vector space of vector space $V$ over a field $k$. When $M$ is a complex manifold of (complex) dimension $n$, then it is an oriented real manifold of dimension $2n$, so applying $( - ) \ot_{\bR} \bC$ to the above isomorphism gives 

$$H^{i}(M) \simeq H^{2n - i}_{c}(M)^{\vee}.$$

\

\subsection{Totaro's work} Our approach to attack Theorem \ref{main2} is to use the Leray spectral sequence of the inclusion $F(X, n) \hra X^{n}$ for a variety $X$, where

$$F(X, n) := \{(x_{1}, \dots, x_{n}) \in X^{n} : x_{i} \neq x_{j} \text{ whenever } i \neq j\}$$

\

so that $\Conf^{n}(X) = F(X, n)/S_{n}$. We are only interested in the case $X = \Ep$. The $E_{2}$ page of this spectral sequence is described by Totaro in \cite{Tot}, and we closely follow \cite[Theorem 3]{Tot}.  For our specific $X = \Ep$, the differentials $d_{r} : E_{r}^{p,q} \ra E_{r}^{p+r, q+1-r}$ are only possibly nonzero at the page $r = 2$, so $E_{\infty}^{p,q} = E_{3}^{p,q}$. Most importantly, We have

$$E_{3}^{p,q} = \Gr_{p+2q}(H^{p+q}(F(X, n))).$$

\

We note that the variety $X$ in Totaro's proof is assumed to be smooth projective, while $X = \Ep$ we use here is not projective despite being smooth. Nevertheless, the assumption that $X$ is smooth and projective is to ensure that $H^{i}(X^{n})$ is pure of weight $i$. Since we have already checked $H^{i}((\Ep)^{n})$ is pure of weight $i$ in Section \ref{implication}, we can apply the same argument. Totaro's description of the $E_{2}$ page is given with respect to the \textbf{diagonal class} $[\Dl] \in H^{2}(X^{2})$ of $X = \Ep$. By definition, this is the image of $1 \in H^{0}(X)$ via the composition

$$H^{0}(X) \simeq H^{2}_{c}(X)^{\vee} \ra H_{c}^{2}(X^{2})^{\vee} \simeq H^{2}(X^{2}),$$

\

where the isomorphisms are given by the Poincar\'e duality and the middle map is given by taking the dual of the pullback $H^{2}_{c}(X^{2}) \ra H_{c}^{2}(X)$ of the diagonal map $\delta : X \hra X^{2}$. Following the maps above, we have

$$1 \mapsto \left([\eta] \mapsto \int_{X} \eta\right) \mapsto \left([\omega] \mapsto \int_{X} \delta^{*}(\omega) \right) = \left([\omega] \mapsto \int_{X^{2}} \Dl \wg \om \right)\mapsfrom [\Dl],$$

\

so the diagonal class $[\Dl]$ is characterized by the conditions

$$\int_{X} \delta^{*}(\om) = \int_{X^{2}} \Dl \wg \om$$

\

where $\om$ varies over all compactly supported $2$-forms on $X^{2}$ so that $[\om] \in H^{2}_{c}(X^{2})$. Our situation $X = \Ep$ is special in the sense that $[\Dl]$ can be explicitly computed in terms of two generators of $H^{1}(X)$. (Note that  $h^{0}(X) = 1$ and $h^{1}(X) = 2$ while $h^{i}(X) = 0$ for all $i \geq 2$.) We set some convenient notation to state this fact: denoting by $p_{1}, p_{2} : X^{2} \ra X$ the two projections, for any $\alpha \in H^{\bl}(X)$, we write $\alpha_{i} := p_{i}^{*}(\alpha) \in H^{\bl}(X^{2})$. 

\

\begin{lem}\label{diag} It is possible to choose a basis $x, y \in H^{1}(X)$ such that

$$[\Dl] = y_{1} \wg x_{2} - x_{1} \wg y_{2} \in H^{2}(X^{2}).$$

\
\end{lem}

\begin{proof} We use the notation $[\Delta_{E}]$ and $[\Delta_{X}]$ to mean the diagonal classes of $E$ and $X$, respectively. First, it is well-known that there are closed 1-forms on $E$ such that their cohomology classes $x, y \in H^{1}(E, \bR) \simeq \bR^{2}$ form a basis for $H^{1}(E, \bR) \simeq \bR^{2}$ and

$$\int_{E} x \wg y =  1.$$

\

We also use $x, y$ to mean the basis $x \ot 1, y \ot 1 \in H^{1}(E, \bR) \ot_{\bR} \bC \simeq \bC^{2}$. These $x, y$ are compactly supported in $E$ because $E$ is compact. We claim that

$$[\Dl_{E}] = x_{1} \wg y_{1} + y_{1} \wg x_{2} - x_{1} \wg y_{2} + x_{2} \wg y_{2}$$

\

on $E$. To check this, one can just check such the right-hand side satisfies the equations

$$\int_{E} \delta^{*}(\om) = \int_{E^{2}} \Dl_{E} \wg \om$$

\

for all $\om \in H^{2}(E^{2})$ in place of $\Dl_{E}$. Note that it is enough to show the identity for the following possibilities of $\om$:

\begin{itemize}
	\item $x_{1} \wedge x_{2}$ and $y_{1} \wedge y_{2}$;
	\item $x_{i} \wg y_{j}$ for $1 \leq i, j \leq 2$,
\end{itemize}

as they form a basis for $H^{2}(E^{2})$. Since $p_{i} \circ \delta = \id_{X}$ for $i = 1, 2$, we have $\delta^{*}(x_{1} \wedge x_{2}) = 0 = \delta^{*}(y_{1} \wedge y_{2})$ and $\delta^{*}(x_{i} \wedge y_{j}) = x \wg y$. On the other hand, we have

\begin{align*}
(x_{1} \wg y_{1} &+ y_{1} \wg x_{2} - x_{1} \wg y_{2} + x_{2} \wg y_{2}) \wg x_{1} \wg x_{2}  \\
&= 0 \\
&= (x_{1} \wg y_{1} + y_{1} \wg x_{2} - x_{1} \wg y_{2} + x_{2} \wg y_{2}) \wg y_{1} \wg y_{2},
\end{align*}

\

and

$$(x_{1} \wg y_{1} + y_{1} \wg x_{2} - x_{1} \wg y_{2} + x_{2} \wg y_{2}) \wg x_{i} \wg y_{j} = x_{1} \wg y_{1} \wg x_{2} \wg y_{2}.$$

\

Since

\begin{align*}
\int_{E^{2}} x_{1} \wg y_{1} \wg x_{2} \wg y_{2}  &= \int_{E^{2}} p_{1}^{*}(x \wg y) \wg p_{2}^{*}(x \wg y) \\
&= \left(\int_{E} x \wg y\right)^{2} \\
&= 1,
\end{align*}

\

this establishes the claim.

\

\hspace{3mm} To finish the proof, we first note that $x_{1} \wg y_{1}$ and $x_{2} \wg y_{2}$ pull back to zero on $X^{2} = (\Ep)^{2}$. Thus, the claim follows from our computation on $E$ because the restriction $H^{2}(E^{2}) \ra H^{2}(X^{2})$ maps the diagonal class of $E$ to that of $X$ (by functoriality of Gysin maps). This finishes the proof. 
\end{proof}

\

\hspace{3mm} Let $E_{r}^{p,q}(X, n)$ the $(p,q)$ component of the $r$-th page of the Leray spectral sequence of $F(X, n) \hra X^{n}$ and

$$E_{r}(X, n) := \bop_{i = 0}^{\infty}\bop_{p + q = i} E^{p,q}_{r}(X, n).$$

\

Totaro's description of $E_{2}(X, n)$ is as follows. We will only consider the case $X = \Ep$, but his results has a generalization for any oriented real manifolds. However, the degeneration at the $E_{3}$ page (i.e., $E_{3}(X, n) \simeq E_{\infty}(X, n)$) is particular to this case (or any variety whose $i$-th cohomology group is pure of weight $i$ for every $i \geq 0$).

\


\begin{prop}[\cite{Tot}, Theorem 1]\label{E2} Keeping the notation as above, we have

$$E_{2}(X, n) = \frac{H^{\bl}(X^{n})[g_{ij} : 1 \leq i \neq j \leq n]}{(\mr{relations})},$$

\

where the right-hand side is a presentation for the bigraded-commutative algebra with the following relations:

\

\be
	\item $g_{ij} = g_{ji}$ for all $1 \leq i \neq j \leq n$;
	\item $g_{ik}g_{jk} = - g_{ij}(g_{ik} + g_{jk})$ for all distinct $1 \leq i, j, k \leq n$;
	\item $\alpha_{i} g_{ij} = \alpha_{j} g_{ij}$ for all $1 \leq i \neq j \leq n$ and $\alpha \in H^{\bl}(X^{n})$,
\ee

\

and the elements of $H^{p}(X^{n})$ get degree $(p, 0)$ while each $g_{ij}$ gets degree $(0, 1)$. The differential 

$$d = d_{2}^{p,q} : E_{2}^{p,q}(X, n) \ra E_{2}^{p+2, q - 1}(X, n)$$

\

is given by sending everything in $H^{\bl}(X^{n})$ to $0$ and

$$d(g_{ij}) = p_{ij}^{*}([\Dl]),$$

\

where $p_{ij} := (p_{i}, p_{j}) : X^{n} \ra X^{2}$. Moreover, the action of $S_{n}$ on $E_{2}(X, n)$ induced by permuting coordinates on $F(X, n) \sub X^{n}$ is described as follows: the $S_{n}$-action on $H^{\bl}(X^{n})$ is induced from the $S_{n}$-action on $X^{n}$ by permuting coordinates and

$$\sigma g_{ij} := g_{\sigma(i), \sigma(j)}$$

\

for $\sg \in S_{n}$.
\end{prop}

\

\begin{rmk}\label{diff} By Lemma \ref{diag}, we necessarily have

$$d(g_{ij}) = y_{i} \wg x_{j} - x_{i} \wg y_{j},$$

\

where $x_{i} := p_{i}^{*}(x)$ and $y_{i} := p_{i}^{*}(y)$, writing $p_{i} : X^{n} \ra X$ to mean the $i$-th projection. To see this, recall that Lemma \ref{diag} says $[\Dl] = y_{1} \wg x_{2} - x_{1} \wg y_{2}$ so that

\begin{align*}
d(g_{ij}) &= p_{ij}^{*}([\Dl]) \\
&= p_{ij}^{*}(y_{1} \wg x_{2}) - p_{ij}^{*}(x_{1} \wg y_{2}) \\
&= p_{ij}^{*}(y_{1}) \wg p_{ij}^{*}(x_{2}) - p_{ij}^{*}(x_{1}) \wg p_{ij}^{*}(y_{2}) \\
&= p_{ij}^{*}(p_{1}^{*}(y)) \wg p_{ij}^{*}(p_{2}^{*}(x)) - p_{ij}^{*}(p_{1}^{*}(x)) \wg p_{ij}^{*}(p_{2}^{*}(y)) \\
&= (p_{1} \circ p_{ij})^{*}(y) \wg (p_{2} \circ p_{ij})^{*}(x) - (p_{1} \circ p_{ij})^{*}(x) \wg (p_{2} \circ p_{ij})^{*}(y) \\
&= p_{i}^{*}(y) \wg p_{j}^{*}(x) - p_{i}^{*}(x) \wg p_{j}^{*}(y) \\
&= y_{i} \wg x_{j} - x_{i} \wg y_{j}.
\end{align*}
\end{rmk}

\

\hspace{3mm} We have

$$E_{3}^{p,q}(X, n) = \Gr_{p+2q}(H^{p+q}(F(X,n))),$$

\

but it also turns out that 

$$E_{3}^{p,q}(X, n)^{S_{n}} = \Gr_{p+2q}(H^{p+q}(F(X,n)/S_{n})) = \Gr_{p+2q}(H^{p+q}(\Conf^{n}(X))),$$

\

and hence we can use the presentation of $E_{2}(X, n)$ described above to approach Theorem \ref{main2}.

\

\section{Proof of Theorem \ref{main2}} \label{proof}

\hspace{3mm} In this section, we prove Theorem \ref{main2}. Again, we work with $X = \Ep$, a punctured elliptic curve over $\bC$.

\

\subsection{Setup and goal}\label{setup} The first step is to note that our problem is entirely algebraic. That is, from the previous section, we have the following description of the graded-commutative $\bC$-algebra:

$$E_{2}(X, n) = \bC
\begin{bmatrix}
x_{1}, \dots, x_{n}, \\
y_{1}, \dots, y_{n}, \\
g_{ij} \text{ for } 1 \leq i \neq j \leq n
\end{bmatrix}/(\text{relations}),$$

\

where the relations are given by 

\be
	\item $g_{ij} = g_{ji}$ for all $1 \leq i \neq j \leq n$;
	\item $g_{ik}g_{jk} = - g_{ij}(g_{ik} + g_{jk})$ for all distinct $1 \leq i, j, k \leq n$;
	\item $g_{ij}x_{i} = g_{ij}x_{j}$ for all $1 \leq i \neq j \leq n$;
	\item $g_{ij}y_{i} = g_{ij}y_{j}$ for all $1 \leq i \neq j \leq n$;
	\item $x_{i}y_{i} = 0$ for all $1 \leq i \leq n$,
\ee

\

and the degrees of $x_{i}, y_{i}$ are $(1, 0)$, while the degree of each $g_{ij}$ is $(0, 1)$.\footnote{When it comes to graded-commutativity, the total degree of an element of degree $(p, q)$ is $p + q$.} By Remark \ref{diff}, the differential 

$$d = d_{2}^{p,q} : E_{2}^{p,q}(X, n) \ra E_{2}^{p+2, q - 1}(X, n)$$

\

is given by

$$d(g_{ij}) = y_{i}x_{j} - x_{i}y_{j}.$$

\

\textbf{Goal}. Proving Theorem \ref{main2} is equivalent to proving that $E_{3}^{p,q}(X, n)^{S_{n}} = 0$ unless $p - q = 0 \text{ or } 1$. Indeed, since $p + q = i$ and $p + 2q = w(i)$, we have $p = 2i - w(i)$ and $q = w(i) - i$ so that $p - q = 3i - 2w(i)$. This implies that $w(i) = (3i - (p-q))/2$.

\

\subsection{Special elements of $E_{2}(X, n)$} Since we are over $\bC$, whose characteristic is $0$, taking the $S_{n}$-invariants is an exact functor, so it commutes with taking cohomlogy. Thus, we may take the cohomology of $E_{2}(X,n)^{S_{n}}$ to compute $E_{3}(X,n)^{S_{n}}$. A typical element of $E_{2}(X, n)$ is a linear combination of the elements of the form

$$g_{i_{1},j_{1}} \cdots g_{i_{a},j_{a}} x_{k_{1}} \cdots x_{k_{b}} y_{l_{1}} \cdots y_{l_{c}}.$$

\

We may assume that all of $k_{1}, \dots, k_{b}, l_{1}, \dots, l_{c}$ are distinct because otherwise such an element is $0$ by graded commutativity or some of the relations we have above. In this section, we provide various lemmas about such elements that will help us prove Theorem \ref{main2}.

\

\begin{lem}\label{rel1} For any $r \geq 2$, we have

$$g_{1,2}g_{2,3} \cdots g_{r-1,r} g_{r,1}  = 0$$

\

in $E_{2}(X, n)$.
\end{lem}

\begin{proof} We proceed by induction. For $r = 2$, we have $g_{1,2}g_{2,1} = g_{1,2}^{2} = 0$ by relation (1) and graded commutativity. For the induction hypothesis, suppose that $g_{1,2}g_{2,3} \cdots g_{r-2,r-1} g_{r-1,1}  = 0$ with $r - 1 \geq 2$. Then

\begin{align*}
g_{1,2}g_{2,3} \cdots g_{r-2,r-1} g_{r-1,r}g_{r,1} &= g_{1,2}g_{2,3} \cdots g_{r-2,r-1} (-g_{r-1,1}(g_{r-1,r}+g_{r,1})) \\
&= -(g_{1,2}g_{2,3} \cdots g_{r-2,r-1}g_{r-1,1})(g_{r-1,r}+g_{r,1}) \\
&= 0
\end{align*}

\

by (2) and the induction hypothesis. This finishes the proof.
\end{proof}

\hspace{3mm} The following notation will be convenient:

\begin{itemize}
	\item $x_{ij} := g_{ij}x_{i} = g_{ij}x_{j}$;
	\item $y_{ij} := g_{ij}y_{i} = g_{ij}y_{j}$;
	\item $g_{I} = g_{i_{1}, \dots, i_{r}} := g_{i_{1}, i_{2}} g_{i_{2}, i_{3}} \cdots g_{i_{r-2}, i_{r-1}} g_{i_{r-1}, i_{r}}$;
	\item $x_{I} := x_{i_{1}, \dots, i_{r}} := g_{i_{1}, \dots, i_{r}}x_{i_{1}} = \cdots = g_{i_{1}, \dots, i_{r}}x_{i_{r}}$;
	\item $y_{I} := y_{i_{1}, \dots, i_{r}} := g_{i_{1}, \dots, i_{r}}y_{i_{1}} = \cdots = g_{i_{1}, \dots, i_{r}}y_{i_{r}},$
\end{itemize}

\

where $I = (i_{1}, \dots, i_{r})$ and $i_{1}, \dots, i_{r}$ are distinct. 

\

\begin{lem}\label{rel2} Any product of $g_{i,j}$ can be written as a linear combination of elements in $E_{2}(X, n)$ of the following form:

$$g_{I_1} g_{I_2} \cdots g_{I_r}$$

\

where $I_1,\dots,I_r$ are disjoint ordered tuples.
\end{lem}

\begin{proof} Given a product of $g_{i,j}$, consider the undirected graph with vertex set $\{1,\dots,n\}$ with $i,j$ connected by an edge if and only if $g_{i,j}$ appears in the product. We discard the vertices that have no edges connected to them for convenience. By Lemma \ref{rel1}, we may assume that this graph is acylic, so it is a disjoint union of trees. Consider the following special case first: $g_{1,2}g_{2,3}g_{2,4}$. We have

\begin{align*}
g_{1,2}g_{2,3}g_{2,4} &= g_{1,2}(-g_{3,4}(g_{2,3} + g_{2,4})) \\
&= -g_{1,2}g_{3,4}g_{2,3} - g_{1,2}g_{3,4}g_{2,4} \\
&= g_{1,2}g_{2,3}g_{3,4} + g_{1,2}g_{2,4}g_{4,3} \\
&= g_{1,2,3,4}+g_{1,2,4,3}.
\end{align*}

\

For the general case, applying this computation repeatedly to each connected component finishes the proof. (There are several ways to show that the process terminates on a given tree. For instance, one can assign a root and induct on the depth of the tree, namely, the length of the longest path from the root.) 
\end{proof}

\

By Lemma \ref{rel2}, we may write any typical element of $E_{2}(X, n)$ as a linear combinations of the elements of the form

$$g_{I_{1}} \cdots g_{I_{a}} x_{J_{1}} \cdots x_{J_{b}} y_{K_{1}} \cdots y_{K_{c}},$$

\

where $I_{1}, \dots, I_{a}, J_{1}, \dots, J_{b}, K_{1}, \dots, K_{c}$ are \emph{disjoint} ordered sets of distinct integers. Since $E_{2}(X, n)^{S_{n}} = e_{S_{n}}(E_{2}(X, n))$, where

$$e_{S_{n}} := \frac{1}{|S_{n}|}\sum_{\sg \in S_{n}} \sg$$

\

is the averaging operator, we can list all the generators of $E_{2}(X, n)^{S_{n}}$ by applying $e_{S_{n}}$ to the elements of the above form. We note that

$$d(x_{ij}) = d(g_{ij})x_{i} - g_{ij}d(x_{i}) = d(g_{ij})x_{i} = (y_{i}x_{j} - x_{i}y_{j})x_{i} = 0$$

\

because $x_{i}y_{i} = 0$ and $x_{i}^{2} = 0$. Similarly, we have $d(y_{ij}) = 0$.

\

\begin{lem}\label{rel3} Let 

$$\alpha = g_{I_{1}} \cdots g_{I_{a}} x_{J_{1}} \cdots x_{J_{b}} y_{K_{1}} \cdots y_{K_{c}}$$

\

be an element of $E_{2}(X, n)$, where $I_{1}, \dots, I_{a}, J_{1}, \dots, J_{b}, K_{1}, \dots, K_{c}$ are disjoint ordered sets of distinct integers.

\

\be
	\item If $\alpha$ can be written in an expression that includes $x_{i}$ and $x_{j}$ for some $1 \leq i \neq j \leq n$, then $e_{S_{n}}(\alpha) = 0$.
	\item If $\alpha$ can be written in an expression that includes $g_{ij}$ and $g_{kl}$ such that $\{i,j\} \cap \{k,l\} = \es$, then $e_{S_{n}}(\alpha) = 0$.
\ee
\end{lem}

\begin{proof} If an expression of $\alpha$ includes $x_{i}$ and $x_{j}$ for some $1 \leq i \neq j \leq n$, then using the graded commtativity finitely many times, we may write $\alpha$ in an expression that includes $x_{i}x_{j}$. Then we consider the transposition $\sg = (i \ j) \in S_{n}$ and see that $e_{S_{n}}(\alpha) = e_{S_{n}}(\sg(\alpha)) = e_{S_{n}}(- \alpha) = -e_{S_{n}}(\alpha)$ so that $e_{S_{n}}(\alpha) = 0$ since we are working over $\bC$, a field of characteristic $0$. This proves the first assertion. The second assertion can be similalry proven by taking $\sg = (i \ k)(j \ l)$ instead.
\end{proof}

\

\begin{lem}\label{rel4} For $r \geq 3$, we have

$$e_{S_{n}}(g_{i_{1}, \dots, i_{r}}) = 0.$$
\end{lem}

\begin{proof} By a result of Cohen (e.g., \cite{Coh}, Section 6), we have a graded isomorphism of graded-commutative $\bC[S_{n}]$-algebras

$$H^{\bl}(F(\bA^{1}, n)) \simeq \bC
\begin{bmatrix}
g_{ij} \text{ for } 1 \leq i \neq j \leq n
\end{bmatrix}/(\text{relations}),$$

\

where the relations are given by

\

\be
	\item $g_{ij} = g_{ji}$ for all $1 \leq i \neq j \leq n$;
	\item $g_{ik}g_{jk} = - g_{ij}(g_{ik} + g_{jk})$ for all distinct $1 \leq i, j, k \leq n$,
\ee

\

and the degree of each $g_{ij}$ is $1$. Cohen's result also says that the $S_{n}$-action induced from the one on $F(\bA^{1}, n) \sub \bA^{n}$, by permuting coordinates, is given by $\sg g_{i,j} = g_{\sg(i),\sg(j)}$. We thus have a map 

$$H^{\bl}(F(\bA^{1}, n)) \ra E_{2}(X, n)$$

\

of $\bC[S_{n}]$-algebras given by $g_{ij} \mapsto g_{ij}$. Taking the $S_{n}$-invariant is an exact functor, so we get

$$H^{\bl}(\Conf^{n}(\bA^{1})) \simeq H^{\bl}(F(\bA^{1}, n))^{S_{n}} \ra E_{2}(X, n)^{S_{n}},$$

\

where the isomorphism is given because $F(\bA^{1}, n) \ra \Conf^{n}(\bA^{1})$ is a finite covering space. Consider the element $e_{S_{n}}(g_{i_{1}, \dots, i_{r}})$ in $H^{r}(\Conf^{n}(\bA^{1}))$. For $r \geq 3$, Arnol'd \cite{Arn} proved that $H^{r-1}(\Conf^{n}(\bA^{1})) = 0$ for all $n \geq 0$. Thus, we must have $e_{S_{n}}(g_{i_{1}, \dots, i_{r}}) = 0$ in $H^{r}(\Conf^{n}(\bA^{1}))$, as well as in $E_{2}(X, n)^{S_{n}}$.
\end{proof}

\

\begin{lem}\label{rel5} Every element of $E_{2}(X, n)^{S_{n}}$ is a linear combination of elements of the form

$$e_n(g_{i_{1}, i_{2}}^{r}x_j^{s_1}y_k^{s_2}x_{J_{1}} \cdots x_{J_{b}} y_{K_{1}} \cdots y_{K_{c}}) \in E_{2}(X, n),$$

\

where $r,s_1,s_2 \in \{0, 1\}$, the tuples $J_{1}, \dots, J_{b}, K_{1}, \dots, K_{c}$ all have size 2, and all the lower indices that \emph{appear} are disjoint. (For example, if $s_1=0$, then the index $j$ does not appear.)
\end{lem}

\

\begin{proof} Since $E_{2}(X, n)^{S_{n}} = e_{S_{n}}(E_{2}(X, n)) \sub E_{2}(X, n)$, by Lemma \ref{rel3}, any element of it can be written as a linear combination of elements of the form $e_{S_{n}}(\alpha)$ such that

$$\alpha = g_{I_{1}} \cdots g_{I_{a}} x_{J_{1}} \cdots x_{J_{b}} y_{K_{1}} \cdots y_{K_{c}},$$

\

where $I_{1}, \dots, I_{a}, J_{1}, \dots, J_{b}, K_{1}, \dots, K_{c}$ are disjoint ordered sets of distinct integers. If any $I_{t}$ has size at least $3$, then $e_{S_{n}}(\alpha) = 0$ because of mutual disjointness of the index sets and Lemma \ref{rel4}. Thus, if there are any $g_{I_{t}}$ in the expression, we may assume that $|I_{t}| = 2$.

\

\hspace{3mm} If $\sg$ is any permutation on some restricted letters, say $1, 2, \dots, r$, then

$$\sg x_{1, 2, \dots, r} = \sg (g_{1, 2,\dots, r}x_{1}) =g_{\sg(1), \sg(2), \dots, \sg(r)}x_{\sg(1)} = x_{\sg(1), \sg(2), \dots, \sg(r)}.$$

\

Thus, we may apply a similar argument as before to be able to assume that

$$|J_{1}|, \dots, |J_{b}|, |K_{1}|, \dots, |K_{c}| \leq 2$$

\

in order for $e_{S_{n}}(\alpha)$ to be nonzero.

\

\hspace{3mm} Finally, by Lemma \ref{rel3}, we have that $e_{S_n}(\alpha)$ is zero unless $a\leq 1$, there is at most one $J_t$ with $|J_t|=1$, and there is at most one $K_t$ with $|K_t|=1$. This finishes the proof.
\end{proof}

\

\begin{rmk} Assume $|I|, |J|, |K| = 2$. In our notation for $g_{I}, x_{J},$ and $y_{K}$, even though the definitions require $I, J,$ and $K$ to be ordered sets of integers, the orders do not matter due to some of the relations we described at the beginning of Section \ref{setup}. Hence, from now on, we may consider $I, J,$ and $K$ as unordered sets of size 2 without ambiguity.
\end{rmk}

\

We will also need the following technical results in the proof of Theorem \ref{main2}.

\begin{lem}\label{rel6} Let $n, b, c \in \bZ_{\geq 0}$ such that $n \geq 2 + b + c$. The elements of the form

$$x_{j}y_{k} x_{J_{1}} \cdots x_{J_{b}} y_{K_{1}} \cdots y_{K_{c}},$$

\

where $\{j, k\}, J_{1}, \dots, J_{b}, K_{1}, \dots, K_{c}$ are disjoint subsets of $\{1, 2, \dots, n\}$ and $|J_{t}| = |K_{u}| = 2$ for $1 \leq t \leq b$ and $1 \leq u \leq c$, with the increasing lexicographic order among the two-element set indices,\footnote{For example, we have $\{1, 2\} < \{3, 6\} < \{4,5\}$.} are linearly independent in $E_{2}(X, n)$.
\end{lem}

\

\begin{cor}\label{rel7} With the notation in Lemma \ref{rel6}, we have 

$$e_{S_{n}}(x_j y_k x_{J_{1}} \cdots x_{J_{b}} y_{K_{1}} \cdots y_{K_{c}}) \neq 0$$

\

in $E_{2}(X, n)^{S_{n}}$.
\end{cor}

\begin{proof}[Proof of Corollary \ref{rel7} given Lemma \ref{rel6}] Since we work over a field of characteristic $0$, it is enough to show that

$$\sum_{\sg \in S_{n}} x_{\sg(j)}y_{\sg(k)}x_{\sg(J_{1})} \cdots x_{\sg(J_{b})}y_{\sg(K_{1})} \cdots y_{\sg(K_{c})} \neq 0.$$

\

This follows from Lemma \ref{rel6} because any elements of the form $x_{J}$ or $y_{K}$ with with two-element indices $J, K$ commute.
\end{proof}

\

\hspace{3mm} We now prove Lemma \ref{rel6}:

\begin{proof}[Proof of Lemma \ref{rel6}] Consider the following graded-commutative $\bC$-algebra 
\[W(n) := E_2(X,n)/(g_{ij}g_{jk}\text{ for }i,j,k\text{ distinct}).\]

Using the discussion in Section \ref{setup}, we see that the algebra $W(n)$ has the following presentation:

$$W(n) = \bC
\begin{bmatrix}
x_{1}, \dots, x_{n}, \\
y_{1}, \dots, y_{n}, \\
g_{ij}\text{ for unordered }i,j\text{ with } 1\leq i\neq j\leq n
\end{bmatrix}/(\text{relations})$$

\

with $\deg(x_{i}) = \deg(y_{i}) = \deg(g_{ij}) = 1$, and $g_{ij}$ and $g_{ji}$ are treated as the same generator\footnote{We are writing $x_{ij} := x_{\{i,j\}}$ to mean the formal variable corresponding to each two-element subset $\{i,j\} \sub \{1, 2, \dots, n\}$.} and the relations are given by

\

\be
	\item $g_{ij}g_{jk} = 0$ for all distinct $1 \leq i, j, k \leq n$;
	\item $g_{ij}x_{i} = g_{ij}x_{j}$ for all $1 \leq i \neq j \leq n$;
	\item $g_{ij}y_{i} = g_{ij}y_{j}$ for all $1 \leq i \neq j \leq n$;
	\item $x_{i}y_{i} = 0$ for all $1 \leq i \leq n$.
\ee

\

To establish a linear independence in $E_2(X,n)$, it is sufficient to check the linear independence in $W(n)$. The strategy is to compare $W(n)$ with the following algebra for which we know a basis.

\

\hspace{3mm} Consider the graded-commutative $\bC$-algebra

$$V(n) := \bC
\begin{bmatrix}
x_{1}, \dots, x_{n}, \\
y_{1}, \dots, y_{n}, \\
x_{ij}, y_{ij} \text{ for unordered } 1 \leq i \neq j \leq n
\end{bmatrix}/(\text{relations})$$

\

with $\deg(x_{i}) = \deg(y_{i}) = 1$ and $\deg(x_{ij}) = \deg(y_{ij}) = 2$, where the relations are generated by any monomial with repeated indices (e.g., $x_{1}y_{12}$ or $x_{1}y_{2}x_{34}x_{56}y_{57}$). We have a map $\phi : V(n) \ra W(n)$ of graded-commutative $\bC$-algebras given by $x_{i} \mapsto x_{i}$, $y_{i} \mapsto y_{i}$, $x_{ij} \mapsto g_{ij}x_{i}$ (which is well-defined because $g_{ij}x_i=g_{ij}x_j$ in $E_2(X,n)$ as well as in $W(n)$), and $y_{ij} \mapsto g_{ij}y_{i}$. Because the relations defining $V(n)$ are generated by monomials, it can be easily checked that a $\bC$-vector space basis for $V(n)$ can be given by all monomials with disjoint indices (up to rearrangement), so to finish our proof, it is enough to show that the map $\phi : V(n) \ra W(n)$ is injective. We will achieve this by constructing a left inverse $\psi : W(n) \ra V(n)$ as $\bC$-vector spaces, although this map will not be a map of $\bC$-algebras.

\

\hspace{3mm} Consider

$$\widetilde{W}(n) := \bC
\begin{bmatrix}
x_{1}, \dots, x_{n}, \\
y_{1}, \dots, y_{n}, \\
g_{ij} \text{ for unordered } 1 \leq i \neq j \leq n
\end{bmatrix},$$

\

a free graded-commutative $\bC$-algebra with $\deg(x_{i}) = \deg(y_{i}) = \deg(g_{ij}) = 1$. As a $\bC$-vector space, a basis of $\widetilde{W}(n)$ can be given by all monomials

$$g_{I_{1}} \cdots g_{I_{a}} x_{j_{1}} \cdots x_{j_{b}} y_{k_{1}} \cdots y_{j_{c}},$$

\

such that 

\begin{itemize}
	\item $I_1,\dots,I_a$ are in the increasing lexicographic order, 
	\item $j_1,\dots,j_b$ are increasing, and 
	\item $k_1,\dots,k_c$ are increasing.
\end{itemize}

\

We define $\widetilde{\psi} : \widetilde{W}(n) \ra V(n)$ by mapping each of such monomials to an element in $V(n)$ according to the following rules, which we mark as (A), (B), and (C):

\

\textbf{(A)} The map $\widetilde{\psi}$ sends the monomial to $0$ if it contains any of the factors of the following forms (possibly after a rearrangement): $g_{ij}g_{jk}, x_{i}y_{i}, g_{ij}x_{i}x_{j}, g_{ij}y_{i}y_{j}, g_{ij}x_{i}y_{j}$. We say such monomials are of \textbf{type A}.

\

\hspace{3mm} Any nonzero monomial that does not fall into any of the above cases satisfies the following: if $g_{ij}$ appears as a factor, then at most one of $x_i, x_j, y_i, y_j$ can appear; if $x_i$ or $y_i$ appears, then at most one $g_I$ such that $i \in I$ can appear. As a result, such a monomial can be uniquely written in the following form (possibly after a rearrangement and a change of sign):

\begin{equation}\label{gen}
(g_{I_1}\cdots g_{I_a}) (x_{j_1}\cdots x_{j_b}) (y_{k_1}\cdots y_{k_c}) [(g_{L_1}x_{l_1})\cdots (g_{L_d}x_{l_d})] [(g_{M_1}y_{m_1})\cdots (g_{M_e}y_{m_e})], 
\end{equation}

\

where $I_1,\cdots, I_a, j_1,\cdots, j_b, k_1,\cdots, k_c, L_1,\cdots, L_d, M_1,\cdots, M_e$ are disjoint while

\begin{itemize}
	\item $I_1,\dots,I_a$ are in the increasing lexicographic order,
	\item $j_1,\dots,j_b$ are increasing, 
	\item $k_1,\dots,k_c$ are increasing,
	\item $L_1,\dots,L_d$ are in the increasing lexicographic order, and
	\item $M_1,\dots,M_e$ are in the increasing lexicographic order,
\end{itemize}

with $l_s\in L_s$ and $m_t\in M_t$ for $1\leq s\leq d, 1\leq t\leq e$. 

\

\textbf{(B)} If $a>0$, then $\widetilde{\psi}$ sends the above monomial to 0. Such a monomial is said to be of \textbf{type B}.

\

\textbf{(C)} If $a=0$, then $\widetilde{\psi}$ sends the above monomial to 
\[(x_{j_1}\cdots x_{j_b}) (y_{k_1}\cdots y_{k_c}) (x_{L_1}\cdots x_{L_d}) (y_{M_1}\cdots y_{M_e})\in V(n).\]

Such a monomial is said to be of \textbf{type C}. This finishes the construction of $\widetilde{\psi} : \widetilde{W}(n) \ra V(n)$.

\
 
\hspace{3mm} We claim that $\widetilde \psi$ factors through a $\bC$-linear map $\psi : W(n) \ra V(n)$. That is, we want to show that the elements of the following forms go to $0$ under $\widetilde{\psi}$:

\

\be
	\item $g_{ij}g_{jk}M$;
	\item $g_{ij}(x_{i} - x_{j})M$;
	\item $g_{ij}(y_{i} - y_{j})M$;
	\item $x_{i}y_{i}M$,
\ee

\

where $M=M(\bs{x}, \bs{y}, \bs{g})$ is a monomial. It is immediate from the rule (A) of the definition of $\widetilde{\psi}$ that 

$$\widetilde{\psi}(g_{ij}g_{jk}M) = 0 = \widetilde{\psi}(x_{i}y_{i}M).$$

\

To show $\widetilde{\psi}(g_{ij}(x_{i} - x_{j})M) = 0$, we may show that 

\begin{equation}\label{claim}
\widetilde{\psi}(g_{ij}x_{i}M) = \widetilde{\psi}(g_{ij}x_{j}M)
\end{equation}

\

If $M$ contains any one of $x_{i}, x_{j}, y_i, y_j$, or $g_I$ with $I\cap \{i,j\}\neq \varnothing$, then both sides are $0$ according to the rule (A), which gives the equality. Thus, we may assume $M$ involves only indices other than $i$ and $j$. Since $g_{ij}x_{i}M$ and $g_{ij}x_{j}M$ are both monomials, if $M$ is a monomial of type A,  then both sides of \eqref{claim} are $0$ according to the rule (A). This will finish our task, so suppose that this is not the case. That is, the monomial $M$ can be written in the form of \eqref{gen} (possibly after a rearrangement and a change of sign):

\[M=(g_{I_1}\cdots g_{I_a}) (x_{j_1}\cdots x_{j_b}) (y_{k_1}\cdots y_{k_c}) [(g_{L_1}x_{l_1})\cdots (g_{L_d}x_{l_d})] [(g_{M_1}y_{m_1})\cdots (g_{M_e}y_{m_e})],\]
where the indices are disjoint from $\{i,j\}$. It follows that both $g_{ij}x_{i}M$ and $g_{ij}x_{j}M$ can be written in the form of \eqref{gen} as
\begin{align*}
g_{ij}x_{i}M&=(g_{I_1}\cdots g_{I_a}) (x_{j_1}\cdots x_{j_b}) (y_{k_1}\cdots y_{k_c}) [(g_{L_1}x_{l_1})\cdots (g_{L_{d^*}}x_{l_{d^*}})(g_{ij}x_i) (g_{L_{d^*+1}}x_{l_{d^*+1}}) \cdots (g_{L_d}x_{l_d})]\\ &[(g_{M_1}y_{m_1})\cdots (g_{M_e}y_{m_e})],
\end{align*}

and
\begin{align*}
g_{ij}x_{j}M&=(g_{I_1}\cdots g_{I_a}) (x_{j_1}\cdots x_{j_b}) (y_{k_1}\cdots y_{k_c}) [(g_{L_1}x_{l_1})\cdots (g_{L_{d^*}}x_{l_{d^*}})(g_{ij}x_j) (g_{L_{d^*+1}}x_{l_{d^*+1}}) \cdots (g_{L_d}x_{l_d})]\\ &[(g_{M_1}y_{m_1})\cdots (g_{M_e}y_{m_e})],
\end{align*}
where $d^*$ is such that $L_1,\cdots,L_{d^*},\{i,j\},L_{d^*+1},\cdots,L_d$ are in the increasing lexicographic order. Note that $g_{ij}x_i$ and $g_{ij}x_j$ have degree 2, so they commute with every element. 

\

\hspace{3mm} If $a>0$, then both sides of \eqref{claim} are 0 according to the rule (B). If $a=0$, then $\widetilde{\psi}$ sends both $g_{ij}x_i M$ and $g_{ij}x_j M$ to
\[(x_{j_1}\cdots x_{j_b}) (y_{k_1}\cdots y_{k_c}) (x_{L_1}\cdots x_{L_{d^*}} x_{ij} x_{L_{d^*+1}} x_{L_d}) (y_{M_1}\cdots y_{M_e})\in V(n)\]
according to the rule (C), so both sides of \eqref{claim} are again equal. This shows $\widetilde{\psi}(g_{ij}(x_{i} - x_{j})M) = 0$, and a similar argument shows $\widetilde{\psi}(g_{ij}(y_{i} - y_{j})M) = 0$. Therefore, we get a map $\psi : W(n) \ra V(n)$, and it is left to show that $\psi\circ \phi$ is the identity map on $V(n)$. It suffices to show that $\psi(\phi(M))=M$ for any monomial $M$ in $V(n)$, but this is immediate.
\end{proof}

\

\subsection{Proof of Theorem \ref{main2}} Due to Lemma \ref{rel5}, we now know that elements of $E_{2}(X, n)^{S_{n}}$ can be written as linear combinations of elements of the form $e_{S_{n}}(\alpha)$, where

$$\alpha = g_{i_{1}, i_{2}}^{r}x_{j}^{s_{1}}y_{k}^{s_{2}}x_{j_{1,1},j_{1,2}} \cdots x_{j_{t_{1},1},j_{t_{1},2}} y_{k_{1,1},k_{1,2}} \cdots y_{k_{t_{2},1},k_{t_{2},2}}  \in E_{2}(X, n),$$

\

with $r, s_{1}, s_{2} \in \{0, 1\}$, and all the numbers appearing as indices are distinct. We note that $\alpha \in E_{2}^{p,q}(X, n)$, where

\begin{align*}
(p, q) &= (0,r) + (s_{1} + s_{2},0) + (2(t_{1} + t_{2}), 2(t_{1} + t_{2})) \\
&= (s_{1} + s_{2} + 2(t_{1} + t_{2}), r + 2(t_{1} + t_{2})),
\end{align*}

\

so $p - q = s_{1} + s_{2} - r$.

\

\begin{proof}[Proof of Theorem \ref{main2}] We continue the discussion above. Our goal is to show that $E_{3}^{p,q}(X, n)^{S_{n}} = 0$ unless $p - q = s_{1} + s_{2} - r \in \{0, 1\}$, where we also recall 

\begin{itemize}
	\item $p = s_{1} + s_{2} + 2(t_{1} + t_{2})$ and
	\item $q = r + 2(t_{1} + t_{2})$.
\end{itemize}

\

Note that

\begin{align*}
d(\alpha) &= d(g_{i_{1}, i_{2}}^{r}x_{j}^{s_{1}}y_{k}^{s_{2}}x_{j_{1,1},j_{1,2}} \cdots x_{j_{t_{1},1},j_{t_{1},2}} y_{k_{1,1},k_{1,2}} \cdots y_{k_{t_{2},1},k_{t_{2},2}}) \\
&= d(g_{i_{1}, i_{2}}^{r})x_{j}^{s_{1}}y_{k}^{s_{2}}x_{j_{1,1},j_{1,2}} \cdots x_{j_{t_{1},1},j_{t_{1},2}} y_{k_{1,1},k_{1,2}} \cdots y_{k_{t_{2},1},k_{t_{2},2}} \\
&= \left\{
	\begin{array}{ll}
	0 & \mbox{if } r=0,  \\
	(y_{i_{1}}x_{i_{2}} - x_{i_{1}}y_{i_{2}})x_{j}^{s_{1}}y_{k}^{s_{2}}x_{j_{1,1},j_{1,2}} \cdots x_{j_{t_{1},1},j_{t_{1},2}} y_{k_{1,1},k_{1,2}} \cdots y_{k_{t_{2},1},k_{t_{2},2}} & \mbox{if } r=1.
	\end{array}\right.
\end{align*}

\


This implies that

\begin{align*}
d(e_{S_{n}}(\alpha)) &= e_{S_{n}}(d(\alpha)) \\
&= \left\{
	\begin{array}{ll}
	e_{S_{n}}((y_{i_{1}}x_{i_{2}} - x_{i_{1}}y_{i_{2}})x_{j_{1,1},j_{1,2}} \cdots x_{j_{t_{1},1},j_{t_{1},2}} y_{k_{1,1},k_{1,2}} \cdots y_{k_{t_{2},1},k_{t_{2},2}}) & \mbox{if } (r,s_{1},s_{2}) = (1,0,0), \\
	0 & \mbox{otherwise},
	\end{array}\right. \\
&= \left\{
	\begin{array}{ll}
	-2e_{S_{n}}(x_{i_{1}}y_{i_{2}}x_{j_{1,1},j_{1,2}} \cdots x_{j_{t_{1},1},j_{t_{1},2}} y_{k_{1,1},k_{1,2}} \cdots y_{k_{t_{2},1},k_{t_{2},2}}) & \mbox{if } (r,s_{1},s_{2}) = (1,0,0), \\
	0 & \mbox{otherwise}.
	\end{array}\right.
\end{align*}

\

By Corollary \ref{rel7}, this shows that $d(e_{S_{n}}(\alpha)) \neq 0$ if and only if $(r, s_{1}, s_{2}) = (1, 0, 0)$. Thus, we see that $\ker(d)$, where $d$ is now the differential for $E_{2}(X,n)^{S_{n}}$, is generated by $e_{S_{n}}(\alpha)$ for $\alpha$ with $(r, s_{1}, s_{2}) \neq (1, 0, 0)$.

\

\hspace{3mm} The table

\begin{center}
\begin{tabular}{ |c|c|c|c| } 
\hline
$(r, s_{1}, s_{2})$ & $p - q = s_{1} + s_{2} - r$ \\
\hline
(0,0,0) & 0 \\ 
(0,0,1) & 1 \\ 
(0,1,0) & 1 \\ 
(0,1,1) & 2 \\
(1,0,0) & -1 \\
(1,0,1) & 0 \\
(1,1,0) & 0 \\
(1,1,1) & 1 \\
\hline
\end{tabular}
\end{center}

\

tells us that the only two choices for $(r, s_{1}, s_{2})$ that give $p - q = s_{1} + s_{2} - r \notin \{0, 1\}$ are $(1, 0, 0)$ and $(0, 1, 1)$.

\

\hspace{3mm} First, let us consider the case $(r, s_{1}, s_{2}) = (1, 0, 0)$, which corresponds to $p - q = -1$. Then $e_{S_{n}}(\alpha) \notin \ker(d)$, and since $\ker(d)$ is bigraded, this implies that $\ker(d^{p,q}) = 0$ for any $p, q$ such that $p - q = -1$. This implies that $E_{3}^{p,q}(X, n)^{S_{n}} = 0$ for any such $p, q$.

\

\hspace{3mm} Next, we consider the case where $\alpha$ has $(r, s_{1}, s_{2}) = (0, 1, 1)$. Then

\begin{align*}
e_{S_{n}}(\alpha) &= -2^{-1}(-2)e_{S_{n}}(x_{j}y_{k}x_{j_{1,1},j_{1,2}} \cdots x_{j_{t_{1},1},j_{t_{1},2}} y_{k_{1,1},k_{1,2}} \cdots y_{k_{t_{2},1},k_{t_{2},2}}) \\
&= -2^{-1} e_{S_{n}}((y_{j}x_{k} - x_{j}y_{k})x_{j_{1,1},j_{1,2}} \cdots x_{j_{t_{1},1},j_{t_{1},2}} y_{k_{1,1},k_{1,2}} \cdots y_{k_{t_{2},1},k_{t_{2},2}}) \\
&= -2^{-1} e_{S_{n}}(d(g_{j, k}x_{j_{1,1},j_{1,2}} \cdots x_{j_{t_{1},1},j_{t_{1},2}} y_{k_{1,1},k_{1,2}} \cdots y_{k_{t_{2},1},k_{t_{2},2}})) \\
&= d( e_{S_{n}}(-2^{-1}(g_{j, k}x_{j_{1,1},j_{1,2}} \cdots x_{j_{t_{1},1},j_{t_{1},2}} y_{k_{1,1},k_{1,2}} \cdots y_{k_{t_{2},1},k_{t_{2},2}}))).
\end{align*}

\

Therefore, we have $[e_{S_{n}}(\alpha)] = 0$ in $E_{3}(X, n)^{S_{n}}$. This shows that $E_{3}^{p,q}(X, n)^{S_{n}} = 0$ unless $p-q \in \{0, 1\}$, as desired.
\end{proof}


\end{document}